\documentclass[a4paper,12pt, leqno]{amsart}
\usepackage{amsmath,amsthm,amsfonts,amssymb,graphicx,mathrsfs,esint,longtable,mathtools, todonotes}
\usepackage[colorlinks=true, linkcolor=red, urlcolor=red, citecolor=green]{hyperref}
\usepackage{cleveref}
\usepackage{color}

\newcommand{\R}{{\mathbb R}}

\def\Xint#1{\mathchoice
{\XXint\displaystyle\textstyle{#1}}%
{\XXint\textstyle\scriptstyle{#1}}%
{\XXint\scriptstyle\scriptscriptstyle{#1}}%
{\XXint\scriptscriptstyle\scriptscriptstyle{#1}}%
\!\int}
\def\XXint#1#2#3{{\setbox0=\hbox{$#1{#2#3}{\int}$ }
\vcenter{\hbox{$#2#3$ }}\kern-.6\wd0}}

\def\dashint{\Xint-}

\newcommand{\mres}{\mathbin{\vrule height 1.6ex depth 0pt width
0.13ex\vrule height 0.13ex depth 0pt width 1.3ex}}
\newcommand{\Ss}{\mathbb{S}}

\newcommand{\Z}{\mathbb{Z}}

\newcommand{\diam}{\textnormal{diam$\,$}}

\newcommand{\Ha}{\mathcal{H}}
\newcommand{\modd}{\textnormal{mod}}
\newcommand{\capp}{\textnormal{cap}}
\newcommand{\degg}{\textnormal{deg\,}}

\def\Aint{- \nobreak \hskip-12.5pt \nobreak \int}

\newtheorem{theorem}{\textbf{THEOREM}}[section]
\newtheorem{lemma}[theorem]{\textsc{Lemma}}
\newtheorem{proposition}[theorem]{\textsc{Proposition}}
\newtheorem{corollary}[theorem]{\textsc{Corollary}}
\theoremstyle{definition}
\newtheorem{assumptions}[theorem]{\textsc{Assumptions}}
{\theoremstyle{remark} }
\def\charfn_#1{{\raise1.2pt\hbox{$\chi_{\kern-1pt\lower3pt\hbox{{$\scriptstyle#1$}}}$}}}
\def\leq{\leqslant }
\def\geq{\geqslant }

\def\Xint#1{\mathchoice
{\XXint\displaystyle\textstyle{#1}}%
{\XXint\textstyle\scriptstyle{#1}}%
{\XXint\scriptstyle\scriptscriptstyle{#1}}%
{\XXint\scriptscriptstyle\scriptscriptstyle{#1}}%
\!\int}
\def\XXint#1#2#3{{\setbox0=\hbox{$#1{#2#3}{\int}$}
\vcenter{\hbox{$#2#3$}}\kern-.5\wd0}}

\def\dashint{\Xint-}

\begin{document}

\title{Duality of moduli in regular toroidal metric spaces}
\author{Atte Lohvansuu} 
\let\thefootnote\relax
\thanks{The author was supported by the Vilho, Yrj\"o and Kalle V\"ais\"al\"a foundation.}
\begin{abstract}We generalize a result of Freedman and He \cite[Th. 2.5.]{FreedmanHe1991}, concerning the duality of moduli and capacities in solid tori, to sufficiently regular metric spaces. This is a continuation of the work of the author and K. Rajala \cite{LohvansuuRajala2018} on the corresponding duality in condensers.
\end{abstract}

\maketitle

\renewcommand{\baselinestretch}{1.2}

\section{Introduction}
Given a metric measure space $(T, d, \mu)$, with $\mu$ Borel-regular, and a collection $\Gamma$ of paths in $T$, the \emph{$p$-modulus} of $\Gamma$ is the number
\[
\modd_p\Gamma:=\inf_{\rho}\int_T\rho^p\, d\mu,
\]
where the infimum is taken over non-negative Borel-functions $\rho$ that satisfy
\begin{equation}\label{eq:testiehtointro}
\int_\gamma\rho\, ds\geq 1
\end{equation}
for all locally rectifiable $\gamma\in\Gamma$. The path modulus is a widely used tool in geometric function theory, especially in connection to quasiconformal mappings \cite{HKST, RickmanQR, Vaisala}.

In the 1960s, F. Gehring \cite{Gehring1962} and W. Ziemer \cite{Ziemer1967} proved that the moduli of paths connecting two compact and connected sets in $\R^n$ are dual to the moduli of surfaces that separate the two sets. The moduli of surface families are defined as above, but instead of condition \eqref{eq:testiehtointro} we require 
\[
\int_S\rho\, d\Ha^{n-1}\geq 1,
\]
where $\Ha^{n-1}$ denotes the $(n-1)$-Hausdorff measure. To describe these duality results in more detail, we need to introduce some notation. Given a connected bounded open subset $G$ of any metric space, and disjoint connected compact sets $E, F\subset G$, denote by $\Gamma(E, F; G)$ the family of paths in $G$ that intersect both $E$ and $F$, and by $\Gamma^*(E, F; G)$ the family of compact subsets of $G$ that separate $E$ and $F$. We say that a set $S$ separates $E$ and $F$ in $G$ if $E$ and $F$ belong to different components of $G-S$. Triples $(E, F, G)$ are called \emph{condensers}. Let $p^*= \frac{p}{p-1}$ be the dual exponent of $1<p<\infty$. By Gehring and Ziemer we then have
\begin{equation}\label{eq:introGZ}
(\modd_p\Gamma(E, F; G))^\frac{1}{p}(\modd_{p^*}\Gamma^*(E, F; G))^\frac{1}{p^*}=1
\end{equation}
in $\R^n$ with $n\geq 2$.

It was shown by the author and K. Rajala that a version of \eqref{eq:introGZ} holds in Ahlfors $q$-regular metric spaces that support a $1$-Poincar\'e inequality. In more detail, a special case of what is shown in \cite{LohvansuuRajala2018} is
\begin{equation}\label{eq:introLR}
\frac{1}{C}\leq (\modd_q\Gamma(E, F; G))^\frac{1}{q}(\modd_{q^*}\Gamma^*(E, F; G))^\frac{1}{q^*}\leq C
\end{equation}
for some constant $C$ that depends only on the data of the space, i.e. the constants that appear in the definitions (see Section \ref{section:results}) of Ahlfors regularity and the Poincar\'e inequalities. Here $E, F$ and $G$ are as in \eqref{eq:introGZ}, and the sets in $\Gamma^*$ are equipped with the $(q-1)$-dimensional Hausdorff measure. 

It should be noted that the inequalities \eqref{eq:introGZ} and \eqref{eq:introLR} are very similar to the reciprocality condition found in \cite{Rajala2017} and \cite{Ikonen2019}. One could also equip the surfaces with the so-called perimeter measures instead of the Hausdorff measure. In this direction a result similar to \eqref{eq:introLR} has recently been proved by Jones and Lahti \cite{JonesLahti2019}.

In this paper we aim to prove a different kind of duality result. Instead of condensers we consider spaces $T$ homeomorphic to the solid torus $\Ss^1\times \mathbb{D}$. It is natural to ask if the duality results above remain valid for the family of paths that go around the 'hole' and the family of surfaces which are bounded by meridians on the boundary torus. It turns out that this is not the case. Freedman and He \cite{FreedmanHe1991} studied conformal moduli on riemannian tori in connection with their research on divergence-free vector fields. They showed that the path-modulus can be arbitrarily small compared to the corresponding surface modulus, even in the smooth setting. However, they managed to prove a duality result by replacing the path modulus with a certain capacity. 

Suppose now that $T$ is equipped with a metric $d$ and a Borel-regular measure $\mu$, so that $(T, d, \mu)$ is Ahlfors $q$-regular. That is, there are constants $a, A>0$ such that 
\[
ar^q\leq \mu(B)\leq Ar^q
\]
for all balls $B$ with radius $r<\mathrm{diam}(T)$. 

Following Freedman and He \cite{FreedmanHe1991} we consider the \emph{degree 1 capacity} instead of the path modulus. It is defined by
\[
\capp_pT:=\inf_{\phi}\int_T\mathrm{Lip}(\phi)^p\, d\mu,
\]
where the infimum is taken over pointwise Lipschitz constants 
\[
\mathrm{Lip}(\phi)(x):=\limsup_{r\rightarrow 0}\sup_{y\in B(x, r)}\frac{|\phi(x)-\phi(y)|}{r}
\] 
of Lipschitz maps $\phi: T\rightarrow \Ss^1$ of degree 1. Loosely speaking, a map is said to have degree 1 if it takes (oriented) loops which generate the corresponding fundamental group to (oriented) generating loops in $\Ss^1$. We assume $\Ss^1$ is equipped with a metric that makes it isometric to a euclidean circle of length 1 equipped with its geodesic metric. 

The surface modulus $\modd_pT$ is defined to be the $p$-modulus of all level sets of continuous functions of degree 1, see Section \ref{section:results}, equipped with the $(q-1)$-dimensional Hausdorff measure. The main results of this paper imply the following.

\begin{theorem}\label{thm:intro}
Let $(T, d, \mu)$ be a compact Ahlfors $q$-regular metric measure space that supports a weak $1$-Poincar\'e inequality. Suppose $T$ is homeomorphic to the solid torus $\Ss^1\times\mathbb{D}$. Let $1<p<\infty$. If $\capp_pT$ is nonzero, then 
\[
\frac{1}{C}\leq (\capp_pT)^{\frac{1}{p}}(\modd_{p^*}T)^\frac{1}{p^*}\leq C
\]
where $C$ is a constant that depends only on the data of $T$. Moreover $\capp_pT=0$ if and only if $\modd_{p^*}T=\infty$.
\end{theorem}
A similar result, with $C=1$, was proved by Freedman and He \cite[Th. 2.5]{FreedmanHe1991} for smooth solid tori equipped with riemannian metrics.

Theorem \ref{thm:intro} is obtained from slightly more general statements. These are Theorems \ref{thm:lowerbound} and \ref{thm:upperbound}, and they correspond to the lower and upper bounds of the inequality in Theorem \ref{thm:intro}, respectively. The proof of the lower bound is essentially the same as the proof of the lower bound of \eqref{eq:introLR} found in \cite{LohvansuuRajala2018}. The main difficulty of the proof of Theorem \ref{thm:intro} is then the upper bound.

In \cite{LohvansuuRajala2018} the proof of the upper bound boils down to showing that given any path $\gamma$ that connects the two continua $E$ and $F$, and a neighborhood $N_\gamma$ of $|\gamma|$, there is a function admissible for the modulus of surfaces separating $E$ and $F$ that is supported in $N_\gamma$. This approach cannot be adopted in our current situation, since the paths have been replaced with Lipschitz maps. Instead, given any level set $S$ of a map of degree 1 and a neighborhood $N_S$ of $S$, we construct a Lipschitz map of degree 1 that is constant outside $N_S$. Note that this implies that the pointwise Lipschitz constant of this map can be assumed to be supported in $N_S$. This approach seems to be new. It can be seen as a dual to the one in \cite{LohvansuuRajala2018}, and as such it can in fact be used to reprove \eqref{eq:introLR}. 

Section \ref{section:results} contains some definitions and the main results. Theorems \ref{thm:lowerbound} and \ref{thm:upperbound} are proved in Sections \ref{section:lowerbound} and \ref{section:upperbound}, respectively.

\vskip 10pt
\noindent
{\bf Acknowledgement.} The author expresses his thanks to the anonymous referee, whose comments led to several improvements.



\section{Main results and definitions}\label{section:results}
For the rest of this text we fix a compact metric measure space $(T, d, \mu)$ that supports a weak $1$-Poincar\'e inequality. We also assume that $\mu$ is doubling. In order to apply the theory of covering spaces later on, we also have to assume that $T$ is semilocally simply connected (local and global path connectedness follow from the $1$-Poincar\'e inequality \cite[8.3.2]{HKST}).

We call a measure $\mu$ \emph{doubling} if it is Borel-regular and there exists a constant $C_\mu>1$, such that for every ball $B=B(x, r)$ with radius $r<\mathrm{diam}(T)$ 
\[
0<\mu(2B)<C_\mu\mu(B)<\infty.
\] 
Here $2B=B(x, 2r)$. 

Let $\mathscr{M}$ be a set of Borel-regular measures on $T$ and let $1\leq p<\infty$. We define the $p$-\emph{modulus} of $\mathscr{M}$ to be
\[
\modd_p\mathscr{M}=\inf\int_T\rho^p\ d\mu,
\]
where the infimum is taken over all Borel measurable functions $\rho: T\rightarrow [0, \infty]$ with 
\begin{equation}\label{eq:modulusehto}
\int_T\rho\ d\nu\geq 1
\end{equation}
for all $\nu\in \mathscr{M}$. Such functions are called \emph{admissible functions of $\mathscr{M}$}. If there are no admissible functions we define the modulus to be infinite. If $\rho$ is an admissible function for $\mathscr{M}-\mathscr{N}$ where $\mathscr{N}$ has zero $p$-modulus, we say that $\rho$ is \emph{$p$-weakly admissible} for $\mathscr{M}$. As a direct consequence of the definitions we see that the $p$-modulus does not change if the infimum is taken over only $p$-weakly admissible functions. If some property holds for all $\nu\in \mathscr{M}-\mathscr{N}$ we say that it holds for $p$-\emph{almost every} $\nu$ in $\mathscr{M}$.

Given a family $\Gamma$ of paths in $T$, the \emph{path $p$-modulus} of $\Gamma$ is denoted and defined like the modulus of a family of measures, but instead of \eqref{eq:modulusehto} it is required that
\[
\int_\gamma\rho\, ds\geq 1
\]
for every locally rectifiable path $\gamma\in\Gamma$.

A Borel function $\rho: T\rightarrow [0, \infty]$ is an \emph{upper gradient} of a function $u: T\rightarrow Y$, where $(Y, d_Y)$ is a metric space, if
\begin{equation}\label{eq:ug}
d_Y(u(\gamma(a)), u(\gamma(b)))\leq \int_\gamma\rho\ ds
\end{equation}
for all rectifiable paths $\gamma: [a, b]\rightarrow T$. The target $Y=[-\infty, \infty]$ is also allowed, but with an additional requirement that the right-hand side of \eqref{eq:ug} has to equal $\infty$ whenever either $|u(\gamma(a))|=\infty$ or $|u(\gamma(b))|=\infty$. If the family of paths for which \eqref{eq:ug} fails has zero $p$-modulus, we say that $\rho$ is a $p$-\emph{weak} upper gradient. The inequality \eqref{eq:ug} is called the \emph{upper gradient inequality} for the pair $(u, \rho)$ on $\gamma$. 

A $p$-integrable $p$-weak upper gradient $\rho$ of $u$ is $\emph{minimal}$ if for any other $p$-integrable $p$-weak upper gradient $\rho'$ of $u$ we have $\rho\leq\rho'$ $\mu$-almost everywhere. By \cite[Theorem 6.3.20]{HKST} minimal $p$-weak upper gradients exist whenever $p$-integrable upper gradients do.

The space $T$ is said to support a \emph{weak} $p$-\emph{Poincar\'e inequality} with constants $C_P$ and $\lambda_P$ if all balls in $T$ have positive and finite measure, and
\[
\dashint_B|u-u_B|\ d\mu\leq C_P\diam(B)\left(\dashint_{\lambda_PB}\rho^p\ d\mu\right)^\frac{1}{p}
\]
for all locally integrable functions $u$ and all upper gradients $\rho$ of $u$. Here 
\[
u_B=\dashint_Bu\, d\mu=\frac{1}{\mu(B)}\int_B u\, d\mu.
\]

In this paper we consider toroidal spaces, meaning that we assume the fundamental group of $T$ to be isomorphic to $\Z$ with respect to any basepoint. Fix a generator $[\alpha_{x_0}]\in\pi_1(T, x_0)$. We say that a loop $\gamma$ with basepoint $x\in T$ is a \emph{degree 1 loop} if it is loop-homotopic to $\alpha_x=\gamma_{xx_0}*\alpha_{x_0}*\overset{\leftarrow}{\gamma}_{xx_0}$ for some path $\gamma_{xx_0}$ that starts at $x$ and ends at $x_0$. It can be shown that the equivalence class $[\alpha_x]\in\pi_1(T, x)$ does not depend on the choice of $\gamma_{xx_0}$. 

For every continuous map $f: T\rightarrow \R/\Z$ there is a unique integer $\degg f$, called the \emph{degree} of $f$, so that for every $x\in T$ and every degree 1 loop $\gamma$ based at $x$ the push-forward $f_*\gamma=f\circ \gamma$ is loop-homotopic to $[f(x)]+\degg f\cdot\beta$, where $\beta: [0, 1]\rightarrow\R/\Z$ is the path $\beta(t)=[t]$. 

Now let $1< p < \infty$. We define the \emph{degree 1 $p$-capacity} of $T$ to be the number
\[
\capp_pT:=\inf\int_T\rho_f^p\, d\mu,
\]
where the infimum is taken over all Lipschitz maps $f: T\rightarrow \R/\Z$ with $\degg f=1$, and $\rho_f$ denotes the minimal $p$-weak upper gradient of $f$. Note that for Lipschitz maps the minimal upper gradient agrees almost everywhere with the pointwise Lipschitz constant $\mathrm{Lip}(f)$, see \cite{Cheeger1999} and \cite[13.5.1]{HKST}. We assume here and hereafter that $\R/\Z$ is equipped with the metric
\[
|[x]-[y]|=\inf_{a\in\Z}|x+a-y|,
\]
where the equivalence classes of $\R/\Z$ are denoted by brackets. Observe that with this metric $\R/\Z$ is isometric to a 1-dimensional euclidean sphere of total length 1 equipped with its intrinsic length metric.

Denote by $\Gamma^*$ the family of all level sets $\phi^{-1}[0]$ with finite codimension 1 spherical Hausdorff measure, where $\phi: T\rightarrow \R/\Z$ is a continuous map of degree 1. The \emph{codimension 1 spherical Hausdorff measure} is defined by 
\[
\Ha(A):=\sup_{\delta>0}\Ha_\delta(A),
\]
where
\[
\Ha_\delta(A):=\inf\sum_i\frac{\mu(B_i)}{r_i},
\]
and the infimum is taken over countable covers $\{B_i\}$ of $A$ by balls with radii $r_i\leq\delta$. By the Carath\'eodory construction $\Ha$ is a Borel-regular measure. A simple application of a coarea estimate, see Proposition \ref{prop:coarea}, shows that almost all level sets of Lipschitz maps have finite $\Ha$-measure. On the other hand, the relative isoperimetric inequality (Lemma \ref{lemma:riie}) shows that level sets of Lipschitz maps of degree 1 must have nonzero $\Ha$-measure.

As a dual counterpart to $\capp_pT$ we consider the surface modulus of $\Gamma^*$. We abbreviate
\begin{equation}\label{eq:pintamod}
\modd_{p^*}T=\modd_{p^*}\lbrace \Ha\mres S\ |\ S\in \Gamma^*\rbrace. 
\end{equation}

The definitions of $\capp_pT$ and $\modd_{p^*}T$ are rather trivial if Lipschitz maps of degree 1 do not exist. Although path-connected topological spaces with fundamental groups isomorphic to $\Z$ can fail to admit maps of nonzero degree, it seems to be unknown whether the existence of such a map is implied by the additional structure of $(T, d, \mu)$. To make life easier we simply assume that there exists at least one Lipschitz map $f: T\rightarrow \R/\Z$ of degree 1. 

Let us gather all of the assumptions into one place for clarity and future reference. 
\begin{assumptions}\label{ass}
The metric measure space $(T, d, \mu)$ is doubling and supports a weak $1$-Poincar\'e inequality. The space $T$ is compact and semilocally simply connected. The fundamental group of $T$ with respect to any basepoint is isomorphic to $\Z$ and there exists at least one Lipschitz map $\phi: T\rightarrow \R/\Z$ of degree 1. 
\end{assumptions}
With these assumptions our main results are the following
\begin{theorem}\label{thm:lowerbound}
Let $1<p<\infty$. If $\capp_pT >0$, then 
\[
\frac{1}{C}\leq(\capp_pT)^\frac{1}{p}(\modd_{p^*}T)^\frac{1}{p^*}, 
\]
where the constant $C$ depends only on the data of $T$. If $\capp_pT =0$, then ${\modd_{p^*}T=\infty}$.
\end{theorem}
\begin{theorem}\label{thm:upperbound}
Let $1<p<\infty$. If $\modd_{p^*}\Gamma^*< \infty$, then 
\[
(\capp_pT)^\frac{1}{p}(\modd_{p^*}T)^\frac{1}{p^*}\leq C, 
\]
where the constant $C$ depends only on the data of $T$. If $\modd_{p^*}T= \infty$, then ${\capp_pT=0}$.
\end{theorem}
We say that a constant $C>0$ depends only on the data of $T$, denoted $C=C(T)$, if it depends only on the constants $C_\mu, C_P$ and $\lambda_P$ appearing in the definitions of doubling measures and Poincar\'e inequalities. The same symbol $C$ will be used for various different constants. 

If we let the metric measure space $(T, d, \mu)$ be as in Theorem \ref{thm:intro}, it satisfies Assumptions \ref{ass}. The existence of Lipschitz maps of degree 1 follows from Proposition \ref{prop:degree}. In Ahlfors $q$-regular spaces the $\Ha$-measure is comparable to the $(q-1)$-dimensional Hausdorff measure, so the surface moduli defined using either measure are comparable. Therefore Theorem \ref{thm:intro} is just a combination of Theorems \ref{thm:lowerbound} and \ref{thm:upperbound}.

Note that the conclusions in Theorems \ref{thm:lowerbound} and \ref{thm:upperbound} are invariant under biLipschitz changes of metrics. Also recall that a complete metric space supporting a Poincar\'e inequality is $C$-quasiconvex for some $C=C(T)$. This means that the change of metrics $(T, d)\rightarrow (T, d')$ is $C$-biLipschitz, when $d'$ is the intrinsic length metric induced by $d$. It follows that we may assume without any loss of generality that $d$ is the length metric. It is then implied by compactness that $(T, d)$ is in fact geodesic. Note that in geodesic spaces we can choose $\lambda_P=1$. For these facts see Theorem 8.3.2 and Remark 9.1.19 in \cite{HKST}.

\section{Proof of Theorem \ref{thm:lowerbound}}\label{section:lowerbound}
The proof of Theorem \ref{thm:lowerbound} is exactly the same as the proof of Theorem 3.1 in \cite{LohvansuuRajala2018}, but with a different coarea estimate.
\begin{proposition}\label{prop:coarea}
Let $u: T\rightarrow \R/\Z$ be Lipschitz and let $\rho$ be a $p$-integrable upper gradient of $u$ in $T$. Let $g: T\rightarrow [0, \infty]$ be a $p^*$-integrable Borel function. Then
\begin{equation}\label{eq:coarea}
\int_{\R/\Z}^*\int_{u^{-1}(t)}g\ d\Ha dt\leq C\int_{T}g\rho\ d\mu
\end{equation}
for some $C=C(T)$.
\end{proposition}
Proposition \ref{prop:coarea} follows by applying \cite[Prop. 4.1]{LohvansuuRajala2018} in small enough balls.

\begin{proof}[Proof of Theorem \ref{thm:lowerbound}]
First assume that $\capp_{p}T>0$. If $\modd_{p^*}T=\infty$, there is nothing to prove. Otherwise let $g\in L^{p^*}(T)$ be admissible for $\modd_{p^*}T$. Let $u: T\rightarrow \R/\Z$ be Lipschitz with degree 1 and note that $u$ must be surjective. Let $\rho$ be an upper gradient of $u$. We may assume that $\rho$ is $p$-integrable. Note that by (\ref{eq:coarea}) $\Ha(u^{-1}(t))<\infty$ for almost every $t$. Proposition \ref{prop:coarea} and H\"older's inequality give
\[
1\leq \int_{\R/\Z}^*\int_{u^{-1}(t)}g\ d\Ha dt\leq C\int_Tg\rho\ d\mu\leq C\left(\int_Tg^{p^*}d\mu \right)^\frac{1}{p^*}\left(\int_T\rho^{p}\,d\mu \right)^\frac{1}{p}.
\]
The lower bound follows by taking infima over admissible functions $g$ and $\rho$. The same argument would lead to a contradiction if $\modd_{p^*}T$ was finite when $\capp_pT=0$.
\end{proof}


\section{Proof of Theorem \ref{thm:upperbound}}\label{section:upperbound}
Theorem \ref{thm:upperbound} follows, once we have shown that there is a non-negative Borel function $\rho_0$ defined on $T$, such that
\[
\capp_pT=\int_T\rho_0^p\, d\mu,
\]
and that
\begin{equation}\label{eq:upperbound}
\capp_pT\leq C(T)\int_{S}\mathcal{M}_{C(T)/n}(\rho_0^{p-1})\, d\Ha
\end{equation} 
for all $S\in\Gamma^*$ and all large enough $n$, depending on $S$. Here $\mathcal M_r$ for $r>0$ denotes the restricted Hardy-Littlewood maximal operator, see \cite[Chapter 3.5]{HKST} for its definition and basic properties. Indeed, letting $n\rightarrow\infty$ and applying the general Fuglede's lemma \cite[Theorem 3]{Fuglede1957} we find that
\begin{equation}\label{eq:kakka1}
\capp_pT\leq C(T)\int_S\rho_0^{p-1}\, d\Ha
\end{equation}
for $\modd_{p^*}$-almost every $S$. Now suppose $\modd_{p^*}T<\infty$. If $\capp_pT=0$, there is nothing to prove. Otherwise it follows from \eqref{eq:kakka1} that the function
\[
\frac{C(T)}{\capp_pT}\rho_0^{p-1}
\]
is weakly admissible for $\modd_{p^*}T$. Thus
\[
(\modd_{p^*}T)^{1/p^*}\leq \frac{C(T)}{\capp_pT}\left(\int_T\rho_0^{p^*(p-1)}\, d\mu\right)^{1/p^*}=C(T)(\capp_pT)^{-1/p}.
\]
The same calculation shows that $\modd_{p^*}T$ must be finite if $\capp_pT$ is nonzero. This proves Theorem \ref{thm:upperbound}. The rest of this section is focused on finding $\rho_0$ and proving \eqref{eq:upperbound}.

Let us begin by constructing $\rho_0$. We would like to apply the usual method of constructing minimizers for capacities or moduli. This method would consist of picking a minimizing sequence $(\phi_i)_i$ of Lipschitz maps of degree 1 and their upper gradients $(\rho_i)_i$, applying weak compactness properties of $L^p$-spaces and Mazur's lemma to find a subsequence of convex combinations of $\rho_i$ that converges strongly to some limit $\rho_0$, and finally showing that $\rho_0$ is an upper gradient of a Lipschitz map of degree 1. The obvious flaw with this method is that it is not clear whether the proposed minimizer $\rho_0$ or the convex combinations of the functions $\rho_i$ are upper gradients of Lipschitz maps of degree 1. 

To fix this, we replace the collection of upper gradients of degree 1 Lipschitz maps by a slightly larger collection $\mathcal{F}$ and show in Proposition \ref{prop:toinenkapasiteetti} that the capacity does not change if we take the infimum over functions of $\mathcal{F}$ instead. The collection $\mathcal{F}$ is defined using the universal cover $(\tilde T, \pi)$ of $T$, and consists of those non-negative Borel functions $\rho$ on $T$ for which the function $\rho\circ\pi$ is an upper gradient of a Newtonian map, which satisfies an analogue of the degree 1 -property. See Subsection \ref{subsec:minimizers} for the definition of Newtonian maps. Once we have set the proper definition of $\mathcal{F}$, it is easy to see that it is convex, and by applying the proofs of existing compactness results on Newtonian spaces we show in Proposition \ref{prop:minimizer} that the limit $\rho_0$ is a member of $\mathcal{F}$ as well.  
\subsection{Universal cover and lifts}We denote the universal cover of $T$ by $(\tilde T, \pi)$. The metric $\tilde d$ on $\tilde T$ is defined as the path metric induced by pulling back the length functional of $T$ with $\pi$. This means that given points $\tilde x, \tilde y\in\tilde T$ we define
\[
\tilde d(\tilde x, \tilde y)=\inf_\gamma \ell(\pi\circ\gamma),
\]
where the infimum is taken over all paths in $\tilde T$ that connect $\tilde x$ and $\tilde y$, and $\ell(\pi\circ\gamma)$ is the length of the path $\pi\circ\gamma$. With this metric $\pi$ becomes a local isometry.

We equip $\tilde T$ with the Borel-regular measure $\tilde\mu$ that satisfies
\[
\tilde\mu(A):=\int_{\pi(A)}N(x, \pi, A)\, d\mu(x),
\]
for all Borel sets $A\subset \tilde T$. Here $N(x, \pi, A)$ denotes the cardinality of $\pi^{-1}(x)\cap A$. The area formula
\[
\int_{\tilde T}f\, d\tilde\mu=\int_T\sum_{y\in \pi^{-1}(x)}f(y)\, d\mu(x)
\]
holds for every integrable Borel-function $f$.

Denote by $\tau: \tilde T\rightarrow\tilde T$ the unique deck transform that satisfies
\[
\tau(\tilde\gamma(0))=\tilde\gamma(1),
\]
for all lifts $\tilde\gamma: [0, 1]\rightarrow \tilde T$ of all degree 1 loops $\gamma: [0 ,1]\rightarrow T$. With the additional metric and measure theoretic structure the classic lifting theorems imply the following.
\begin{lemma}\label{lemma:lipschitzlift}
Suppose $f: T\rightarrow \R/\Z$ is a Lipschitz map of degree $1$ and let $\rho$ be one of its upper gradients. There exists a function $\tilde f: \tilde T\rightarrow \R$, called the \emph{lift} of $f$, that satisfies the following properties.
\begin{enumerate}
\item $[\tilde f]=f\circ\pi$. In particular $\tilde f$ is locally Lipschitz
\item $\rho\circ\pi$ is an upper gradient of $\tilde f$
\item $\tilde f\circ\tau - \tilde f=1$.
\end{enumerate}
Moreover, if $\tilde f'$ is another lift that satisfies the properties above, then there is a $k\in\Z$ such that $\tilde f'=\tilde f\circ\tau^k=\tilde f+k$.
\end{lemma}
Claim (2) follows from the identity
\[
\int_\gamma\rho\circ\pi\, ds=\int_{\pi\circ\gamma}\rho\, ds,
\]
which holds for every rectifiable path $\gamma$ in $\tilde T$.

Conversely, we have the following. 
\begin{lemma}\label{lemma:projektiolemma}
For every locally Lipschitz $g:\tilde T\rightarrow \R$ with $g\circ\tau-g=1$ there is a Lipschitz map $f: T\rightarrow \R/\Z$ of degree 1, that satisfies $[g]=f\circ\pi$. Moreover, if $\rho_f$ is the minimal $p$-weak upper gradient of $f$ in $T$, then $\rho_f\circ\pi$ is the minimal $p$-weak upper gradient of $g$ in $\tilde T$.
\end{lemma}
\begin{proof}
We define $f$ locally by 
\[
f=[g\circ\pi^{-1}].
\]
Then $f$ is well defined due to the property $g\circ\tau-g=1$. It is certainly locally Lipschitz, has degree $1$, and satisfies $[g]=f\circ\pi$. 

It remains to show the relation between the upper gradients. Given any $x\in \tilde T$ there is a ball $B'$ that contains $x$ and on which $\pi$ is an isometry onto $B=\pi(B')$. Clearly $\rho\circ\pi|_{B'}^{-1}$ is a $p$-weak upper gradient of $f$ in $B$ whenever $\rho$ is a $p$-weak upper gradient of $g$ in $B'$. Thus, if $\rho_f\circ\pi$ is a $p$-weak upper gradient of $g$ in $\tilde T$, it must be the minimal one.

Now let $\gamma: [0, 1]\rightarrow B'$ be a rectifiable path, so that the upper gradient inequality holds for the pair $(f, \rho_f)$ on every subpath of $\pi\circ \gamma$. Almost every path in $B'$ is such a path, since $\rho_f$ is a $p$-weak upper gradient of $f$, and as an isometry $\pi|_{B'}$ preserves all path moduli. Continuity of $g$ implies that we can decompose $\gamma$ into $\gamma=\gamma_1*\cdots *\gamma_k$, so that $\gamma_i=\gamma|_{[t_i, t_{i+1}]}$ and
\begin{equation}\label{eq:polkuequ1}
|g(\gamma_i(t_{i+1}))-g(\gamma_i(t_i))|=|[g(\gamma_i(t_{i+1}))]-[g(\gamma_i(t_i))]|
\end{equation}
for all $i=1, \ldots, k$. On these subpaths we have 
\begin{align*}
\int_{\gamma_i}\rho_f\circ\pi\, ds&=\int_{\pi\circ\gamma_i}\rho_f\, ds \\
&\geq |f(\pi(\gamma_i(t_{i+1})))-f(\pi(\gamma_i(t_i)))| \\
&=|[g(\gamma_i(t_{i+1}))]-[g(\gamma_i(t_i))]|.
\end{align*}
Combining this with triangle inequality and \eqref{eq:polkuequ1} yields
\[
|g(\gamma(1))-g(\gamma(0))|\leq \int_\gamma\rho_f\circ\pi\, ds.
\]

Given an open set $U\subset \tilde T$, denote the set of all paths in $U$ on which the upper gradient inequality fails for the pair $(g, \rho_f\circ\pi)$ by $\Gamma_{U}$. We need to show that $\modd_p\Gamma_{\tilde T}=0$. Cover $\tilde T$ by countably many balls $B_i'$, on which $\pi$ is an isometry onto $\pi(B_i')$. Note that if the upper gradient inequality fails for the pair $(g, \rho_f\circ\pi)$ on some path $\eta$, it must fail on some subpath of $\eta$ that is contained in one of the balls $B_i'$. In other words, for every path in the collection $\Gamma_{\tilde T}$ there is a subpath in one of the collections $\Gamma_{B_i'}$. Now
\[
\modd_p\Gamma_{\tilde T}\leq \modd_p\left(\bigcup_{i}\Gamma_{B_i'} \right)\leq \sum_i\modd_p\Gamma_{B_i'}=0,
\]
since the first part of the proof shows that $\modd_p\Gamma_{B_i'}=0$ for all $i$.
\end{proof}
\subsection{Minimizers}\label{subsec:minimizers}
Motivated by Lemmas \ref{lemma:lipschitzlift} and \ref{lemma:projektiolemma} we find an alternative definition for the capacity. 

We say that a function $f: \tilde T\rightarrow \R$ belongs to the Newtonian space $N^{1, p}(\tilde T)$ if $f$ is $p$-integrable and admits a $p$-weak upper gradient that is also $p$-integrable. See \cite[Chapter 7]{HKST} or \cite[Chapter 5]{BjornBjornNonlin} for further properties of these spaces. We say that $f\in N^{1, p}_{loc}(\tilde T)$ if $f|_U\in N^{1, p}(U)$ for every open $U\subset\subset\tilde T$ (note that $\tilde T$ is proper). The space $N^{1, p}(U)$ is equipped with the seminorm
\[
\|f\|_{N^{1, p}(U)}:=\|f\|_{L^p(U)}+\inf_{\rho}\|\rho\|_{L^p(U)},
\] 
where the infimum is taken over all $p$-weak upper gradients $\rho$ of $f$ in $U$.
 
Let $\mathcal{F}$ be the collection of all positive Borel functions $\rho$ on $T$, for which $\rho\circ\pi$ is a $p$-weak upper gradient of some $f\in N^{1, p}_{loc}(\tilde T)$ with $f\circ\tau-f=1$ almost everywhere. Define 
\[
\capp_p^\mathcal{F}T:=\inf_{\rho\in\mathcal{F}}\int_T\rho^p\,d\mu.
\]
Note that by Lemma \ref{lemma:lipschitzlift} every upper gradient of a map admissible for $\capp_pT$ belongs to $\mathcal{F}$. Therefore 
\[
\capp_p^\mathcal{F}T\leq \capp_pT.
\]
The reverse inequality is also valid, but requires a bit more work.
\begin{proposition}\label{prop:toinenkapasiteetti}
\[
\capp_pT=\capp_p^\mathcal{F}T
\]
\end{proposition}
\begin{proof}
We must first show that locally Lipschitz functions of degree 1 are dense in the space of degree 1 functions of $N^{1, p}_{loc}(\tilde T)$. Here having degree 1 means satisfying the property $f\circ\tau-f=1$ almost everywhere. A result by Bj\"orn and Bj\"orn \cite[Th. 8.4.]{BjornBjorn2018} shows that locally Lipschitz functions are dense in $N^{1, p}_{loc}(\tilde T)$. A simple modification of the proof of this result shows that the approximating locally Lipschitz maps can be chosen to be of degree 1 whenever the limit is of degree 1. We provide the main points of this modification.

Following the proof of Theorem 8.4 of \cite{BjornBjorn2018}, we start by choosing for every $x\in \tilde T$ a ball $B_x$ centered at $x$, so that
\begin{itemize}
\item the $1$-Poincar\'e inequality and the doubling property hold within $B_x$, in the sense of \cite{BjornBjorn2018}
\item the covering map $\pi$ is an isometry on $B_x$.
\end{itemize}
Let $U_x:=\pi^{-1}(\pi (\frac{1}{4}B_x))$. The space $T$ is compact, so there is a finite subcollection $\{U_{x_i}\}^m_{i=1}$ that covers $\tilde T$. Write $B_j=\frac{1}{4}B_{x_j}$ and $U_j=U_{x_j}$. Note that $U_j$ can be written as a disjoint union $U_j=\bigcup_{k\in\Z}\tau^kB_j$. We denote $c U_j=\bigcup_{k\in\Z}\tau^k(c B_j)$ for any $c>0$. For each $j$ pick a Lipschitz function $\psi'_j:B_j\rightarrow \R$ that satisfies $\chi_{B_j}\leq \psi'_j\leq \chi_{2B_j}$. Extend these to Lipschitz functions $\psi_j: \tilde T\rightarrow \R$ first by defining $\psi_j|_{\tau^k(2B_j)}:=\psi_j'\circ\tau^{-k}$ in $2U_j$ and then extending as zero to the rest of $\tilde T$. Next, define Lipschitz maps $\varphi_j: \tilde T\rightarrow \R$ recursively with $\varphi_1=\psi_1$ and for $j>1$
\[
\varphi_j=\psi_j\cdot\left(1-\sum_{k=1}^{j-1}\varphi_k\right).
\]
Then $\sum_{k=1}^i\varphi_k=1$ in $U_i$ and $\varphi_j=0$ in $U_i$ for all $j>i$. Therefore $\{ (\varphi_j, U_j)\}_j$ is a partition of unity. 

Now let $f\in N^{1, p}_{loc}(\tilde T)$ be a degree 1 map, $f\circ\tau-f=1$. Let $\varepsilon>0$. By Lemma 8.5 of \cite{BjornBjorn2018} there are locally Lipschitz functions $v_j: 2B_j\rightarrow \R$ with
\[
\|f-v_j\|_{N^{1, p}(2B_j)}\leq \frac{\varepsilon}{1+L_j},
\]
where $L_j$ is the Lipschitz constant of $\varphi_j$. Extend $v_j$ to $2U_j$ with
\[
v_j|_{\tau^k(2B_j)}=k+v_j\circ\tau^{-k}.
\]
Then $v_j\circ\tau-v_j=1$, and for all $k$
\[
\|f-v_j\|_{N^{1, p}(\tau^k(2B_j))}\leq \frac{\varepsilon}{1+L_j}.
\]
As in \cite{BjornBjorn2018} we get 
\begin{equation}\label{eq:jokueq1}
\|\varphi_j(f-v_j)\|^p_{N^{1, p}(\tau^k(2B_j))}\leq 2\varepsilon^p.
\end{equation}
The function $v:=\sum_{j=1}^m\varphi_j v_j$ is locally Lipschitz, and satisfies the degree 1 property $v\circ\tau-v=1$.

Now \eqref{eq:jokueq1} gives
\[
\|v-f\|_{N^{1, p}(U)}\leq C(U)\varepsilon,
\]
for any domain $U\subset\subset \tilde T$. This proves the density of degree 1 locally Lipschitz functions in the space of $N^{1, p}_{loc}(\tilde T)$-functions of degree 1.

Now if we let $\rho\in\mathcal{F}$, the function $\rho\circ\pi$ is a $p$-weak upper gradient of some $f\in N^{1, p}_{loc}(\tilde T)$, and we find a sequence of locally Lipschitz functions $(v_j)$ of degree 1, such that
\[
\|v_j-f\|_{N^{1, p}(U)}\overset{j\rightarrow\infty}{\longrightarrow} 0.
\]
for every $U\subset\subset \tilde T$. Let $w_j$ be the Lipschitz projections of $v_j$, given by Lemma \ref{lemma:projektiolemma}. Then the minimal upper gradients satisfy $\rho_{v_j}=\rho_{w_j}\circ\pi$. Now
\begin{equation}\label{eq:ygeq1}
\|\rho_{v_j}-\rho_f\|_{L^p(B_i)}\overset{j\rightarrow\infty}{\longrightarrow} 0
\end{equation}
for all $i$. Let $A_1=\pi(B_1)$ and for $1\leq j\leq m-1$ define $A_{j+1}:=\pi(B_{j+1})-\bigcup_{i=1}^{j}A_j$. Let $\pi_j: B_j\cap \pi^{-1}(A_j)\rightarrow A_j$ be the restriction of $\pi$ and define $\rho'_f:=\sum_j\chi_{A_j}\rho_f\circ\pi_j^{-1}$. The Borel sets $A_j$ are disjoint and cover $T$, so a quick calculation shows that
\begin{equation}\label{eq:ygeq2}
\|\rho_f'\|^p_{L^p(T)}\leq\|\rho\|^p_{L^p(T)},
\end{equation}
since by definition of $\rho$ we have $\rho_f\leq \rho\circ\pi$ almost everywhere. Finally, note that 
\[
\|\rho_{w_j}-\rho'_f\|^p_{L^p(T)}=\sum_{j=1}^m\|\rho_{w_j}-\rho_f'\|^p_{L^p(A_j)}\leq \sum^m_{j=1}\|\rho_{v_j}-\rho_f\|^p_{L^p(B_j)},
\]
and thus \eqref{eq:ygeq1} implies
\begin{equation}\label{eq:ygeq3}
\lim_{j\rightarrow\infty}\|\rho_{w_j}\|^p_{L^p(T)}=\|\rho_f'\|_{L^p(T)}^p,
\end{equation}
since there are only finitely many sets $A_j$. Combining \eqref{eq:ygeq2} and \eqref{eq:ygeq3} yields
\[
\capp_p T\leq \|\rho\|_{L^p(T)}^p,
\]
which finishes the proof.
\end{proof}
\begin{proposition}\label{prop:minimizer}
There is a minimizer $\rho_0\in\mathcal{F}$, i.e.
\[
\capp_pT=\capp^\mathcal{F}_pT=\int_T\rho_0^p\, d\mu.
\]
Moreover, for any other $p$-integrable $\rho\in\mathcal{F}$ 
\begin{equation}\label{eq:variation}
\capp_pT\leq \int_T\rho_0^{p-1}\rho\, d\mu, 
\end{equation}
where equality holds if and only if $\rho=\rho_0$ almost everywhere.
\end{proposition}
\begin{proof}
Note that $\mathcal{F}$ is convex. Once we know the existence of a minimizer, the proof of the variation inequality \eqref{eq:variation} is standard. See for example \cite[Lemma 5.2.]{LohvansuuRajala2018}. Uniqueness of the minimizer follows from the convexity of $\mathcal{F}$ and the uniform convexity of $L^p(T)$.

We now show the existence of a minimizer. First recall that we have assumed in Assumptions \ref{ass} that there exists at least one Lipschitz map of degree 1. It follows that $\capp_pT$ is finite. Let $(f_i)_i$ be a sequence of locally Lipschitz maps $f_i: T\rightarrow \R/\Z$ of degree 1, so that for each $i$ the function $\rho_i$ is an upper gradient of $f_i$, and
\[
\capp_pT=\lim_{i\rightarrow\infty}\int_T\rho_i^p\, d\mu.
\]
We claim that the lifts $\tilde f_i$ of the maps $f_i$ can be chosen so that the sequence $(\tilde f_i)$ is $L^p$-bounded in any bounded domain of $\tilde T$. 

To this end, note that the length of any loop-homotopically non-trivial loop $\gamma$ must satisfy
\begin{equation}\label{eq:polunpituus}
\ell(\gamma)\geq c
\end{equation}
for some $c>0$. This is implied by the existence of Lipschitz maps of degree 1.

Let $\{x_i\}_{i=1}^N$ be a $\frac{c}{16}$-net in $T$, where $c$ is the constant from \eqref{eq:polunpituus}. Note that by the net property of $\{x_i\}$ any two balls $B_i:=B(x_i, \frac{c}{8})$ are connected by a chain of balls of the same form. By a chain we mean a sequence of balls, in which adjacent ones have nonempty intersection. The same chaining property holds for the balls $2B_i$, but now additionally we find that the connecting chains $(2B_{i_k})_k$ can be chosen so that for each $k$ there is a ball $B_k'\subset 2B_{i_k}\cap 2B_{i_{k+1}}$ of radius $c/8$.

Note that by \eqref{eq:polunpituus} the balls $2B_i$ are evenly covered. In fact, $\pi$ is an isometry when restricted to any component of $\pi^{-1}(2B_i)$. Fix a component $\tilde B_1$ of $\pi^{-1}(B_1)$. Set $V_1=\tilde B_1$. For $k\geq 1$ we define domains $V_k$ recursively by adding components of $\pi^{-1}(B_i)$ for suitable $B_i$. At step $k+1$ we choose exactly one component of $\pi^{-1}(B_i)$, call it  $\tilde B_i$, to be added to $V_k$ if and only if $\pi^{-1}(B_i)$ intersects $V_k$ and there are no components of $\pi^{-1}(B_i)$ that are contained in $V_k$.

After at most $N$ steps no new balls can be added. Let $V=V_N$. It follows from the construction that $V$ is a bounded domain on which $\pi$ is surjective. It may happen that the previous construction does not define $\tilde B_i$ for all $B_i$. If so, just let $\tilde B_i$ be a component of $\pi^{-1}(B_i)$ that is contained in $V$. Thus
\[
V=\bigcup_{i=1}^N\tilde B_i.
\]
Denote by $2\tilde B_i$ the component of $\pi^{-1}(2B_i)$ that contains $\tilde B_i$.

By adding integers if necessary, we may now fix the lifts $\tilde f_i$ by requiring
\begin{equation}\label{eq:liftoletus}
0\leq (\tilde f_i)_{2\tilde B_1}<1.
\end{equation}
If $j\neq 1$, by construction there is a chain $(2\tilde B_{j_k})_{k=1}^l$ with $j_1=1$, $j_l=j$ and $l\leq N$, so that for every $1\leq k<l$ there is a ball $\tilde B_k'\subset 2\tilde B_{j_k}\cap 2\tilde B_{j_{k+1}}$ of radius $c/8$. Let $m:=\min\{\mu(B_i)\}>0$. By the Poincar\'e inequality and the doubling condition
\[
|(\tilde f_i)_{2\tilde B_{j_k}}-(\tilde f_i)_{\tilde B_k'}|\leq C\Aint_{2\tilde B_{j_k}}|\tilde f_i-(\tilde f_i)_{2\tilde B_{j_k}}|\, d\tilde\mu\leq C\|\rho_i\|^p_{L^p(2 B_{j_k})},
\]
where $C=C(T, p, m, c)$, and the same calculation shows 
\[
|(\tilde f_i)_{2\tilde B_{j_{k+1}}}-(\tilde f_i)_{\tilde B_k'}|\leq C\|\rho_i\|^p_{L^p(2B_{j_{k+1}})}
\]
as well. Thus by the triangle inequality and \eqref{eq:liftoletus}
\[
|(\tilde f_i)_{2\tilde B_j}|\leq CN\|\rho_i\|^p_{L^p(T)}+1.
\]
Now by the Sobolev-Poincar\'e inequality, see \cite[Thm. 9.1.2]{HKST}, and the local isometry of $\pi$
\[
\int_{2\tilde B_j}|\tilde f_{i}-(\tilde f_{i})_{2\tilde B_j}|^p\, d\tilde\mu\leq C(T, p, m, c)\|\rho_i\|_{L^p(T)}.
\]
It follows that the sequence $(\tilde f_{i})_i$ is bounded in $L^p(2\tilde B_j)$. 

Since $V$ is covered by finitely many balls $2B_j$, we find that both sequences $(\tilde f_i)_i$ and $(\rho_i\circ\pi)_i$ are bounded in $L^p(V)$, and also in every $L^p(W_k)$, where 
\[
W_k:=\bigcup_{l=-k}^k \tau^lV.
\]
Note that $W_0=V$. Now by extracting enough subsequences we may assume that  $(\tilde f_i)_i$ and $(\rho_i\circ\pi)_i$ converge weakly to functions $\tilde f^0$ and $\tilde \rho^0$ in $L^p(W_0)$. By Lemma 3.1 of \cite{KallunkiShan2001} there exist sequences of convex combinations $(\tilde f^0_k)$ and $(\tilde \rho^0_k)$ of the functions $\tilde f_i$ and $\rho_i\circ\pi$, respectively, that converge strongly to $\tilde f^0$ and $\tilde \rho^0$. Moreover $\tilde\rho^0$ is a $p$-weak upper gradient of $\tilde f^0$ in $W_0$.  

This allows us to define sequences $(\tilde f_i^{k+1})$ and $(\tilde \rho^{k+1}_i)$ recursively to be the sequences in $L^p(W_{k+1})$ that are obtained by applying the argument above on $W_{k+1}$ instead of $W_0$ and on sequences $(\tilde f_i^{k})$ and $(\tilde \rho^{k}_i)_i$ instead of $(\tilde f_i)_i$ and $(\rho_i\circ\pi)_i$. Let $\tilde f^{k+1}$ and $\tilde\rho^{k+1}$ be the corresponding limits in $L^p(W_{k+1})$. It follows that $\tilde f^{k+1}|_{\Omega_k}=\tilde f^{k}$ and $\tilde \rho^{k+1}|_{\Omega_k}=\tilde \rho^{k}$. Define $\tilde f$ and $\tilde\rho: \tilde T\rightarrow \R$ by setting $\tilde f|_{W_k}=\tilde f^k$ and $\tilde \rho|_{W_k}=\tilde \rho^k$. It is immediate that $\tilde\rho$ is a $p$-weak upper gradient of $\tilde f$. 

Consider the diagonal sequences $(\tilde f^j_j)_j$ and $(\tilde \rho^j_j)_j$. These maps are still convex combinations of the functions $\tilde f_i$ and $\rho_i\circ\pi$, respectively. It follows that these sequences converge to $\tilde f$ and $\tilde \rho$ in $L^p_{loc}(\tilde T)$. Moreover $\tilde f\circ\tau -\tilde f=1$ and $\tilde\rho\circ\tau-\tilde\rho=0$ almost everywhere, since these hold everywhere for all maps in the respective sequences. The latter equality allows us to define $\rho_0$ by projecting $\tilde\rho$. Therefore $\rho_0\in\mathcal{F}$ and
\[
\capp^\mathcal{F}_pT=\int_T\rho_0^p\, d\mu,
\]
since $(\tilde \rho^j_j)$ is still a minimizing sequence, due to convexity of $\mathcal{F}$.
\end{proof}
\subsection{Competing admissible maps}
Now that the minimizer $\rho_0$ has been found, the proof of Theorem \ref{thm:upperbound} is only missing the proof of $\eqref{eq:upperbound}$. Recall that \eqref{eq:upperbound} says that for all $S\in\Gamma^*$ 
\[
\capp_pT\leq C(T)\int_{S}\mathcal{M}_{C(T)/n}(\rho_0^{p-1})\, d\Ha,
\]
where $\mathcal{M}$ denotes the Hardy-Littlewood maximal operator. Given an $S\in\Gamma^*$ we construct suitable Lipschitz maps of degree 1 that are constant outside a small neighborhood of $S$. Then we can apply the variation inequality \eqref{eq:variation} of Proposition \ref{prop:minimizer} on the upper gradients of these Lipschitz maps to conclude \eqref{eq:upperbound}.

In this subsection we construct these Lipschitz maps. It turns out that the same construction can be used to obtain Lipschitz maps of degree 1 out of general (continuous) maps of any nonzero degree. We only need to consider maps of positive degree by composing with the antipodal map of $\R/\Z$ if necessary.

To simplify the notation, we omit some parentheses and write for example $\phi^{-1}[0]$ and $\pi\tilde\phi^{-1}(0)$ instead of $\phi^{-1}([0])$ and $\pi(\tilde\phi^{-1}(0))$ from now on.
\begin{proposition}\label{prop:degree}
Let $\phi: T\rightarrow \R/\Z$ be a continuous map of nonzero positive degree. There is a number $N=N(\phi)$, such that for all $n\geq N$ there is a finite pairwise disjoint collection of balls $\{B_i\}$ of radius $1/n$ in $T$, such that for all $i$
\[
\Ha(\phi^{-1}[0]\cap B_i)\geq C(T)n\mu(B_i)
\]
and such that the Borel function
\[
\rho=n\sum_{i}\chi_{5B_i} 
\]
is an upper gradient of a Lipschitz map $\psi: T\rightarrow \R/\Z$ of degree 1.
\end{proposition}
\begin{proof}[Proof of \eqref{eq:upperbound} assuming Proposition \ref{prop:degree}]
Let $S\in \Gamma^*$. Then $S=\phi^{-1}[0]$ for some degree $1$ map $\phi$. Let $\{B_i\}$ be the collection of balls and let $\rho$ be the Borel function that is obtained by applying Proposition \ref{prop:degree} for some large enough $n$. Now
\[
\Ha(S\cap B_i)\geq C(T)n\mu(B_i)
\]
for all $i$. Applying this along with the variation inequality \eqref{eq:variation} of Corollary \ref{prop:minimizer}, the doubling property of $\mu$ and the definition of the Hardy-Littlewood maximal operator gives 
\begin{align*}
\capp_pT&\leq \int_T\rho\rho_0^{p-1}\ d\mu \\
&\leq C(T)\sum_{i}n\mu(B_i)\dashint_{5B_i}\rho_0^{p-1}\ d\mu \\
&\leq C(T)\sum_{i}\Ha(S\cap B_i)\inf_{x\in B_i}\mathcal{M}_{C(T)/n}(\rho_0^{p-1})(x)\\
&\leq C(T)\int_{S}\mathcal{M}_{C(T)/n}(\rho_0^{p-1})\, d\Ha,
\end{align*}
which is exactly \eqref{eq:upperbound}.
\end{proof}
The rest of the section is focused on proving Proposition \ref{prop:degree}. Let $\phi: T\rightarrow \R/\Z$ be a continuous map of nonzero positive degree. Let $x_0\in\phi^{-1}[0]$, $\tilde x_0\in\pi^{-1}(x_0)$ and let $\tilde\phi: \tilde T\rightarrow \R$ be the lift of $\phi$ that satisfies $\tilde\phi(\tilde x_0)=0$. Compactness of $T$ implies that 
\begin{equation}\label{eq:delta}
\delta:=\min\left\lbrace d(\pi\tilde\phi^{-1}(\pm\frac{1}{8}), \pi\tilde\phi^{-1}(0)), d(\pi\tilde\phi^{-1}(\pm\frac{1}{4}), \pi\tilde\phi^{-1}(\pm\frac{1}{8}))\right\rbrace
\end{equation}
is strictly positive. Denote $U^+=\pi\tilde\phi^{-1}(0, 1/4)$ and $U^-=\pi\tilde\phi^{-1}(-1/4, 0]$. Denote also $S=\pi\tilde\phi^{-1}(0)$. Observe that $S\subset \phi^{-1}[0]$, and if $\degg\phi=1$, then $S=\phi^{-1}[0]$. 

For our intents and purposes the relative isoperimetric inequality takes the following form. 
\begin{lemma}(Relative isoperimetric inequality)\label{lemma:riie} \\
There are constants $C=C(T)$ and $\lambda=\lambda(T)\geq 1$ such that 
\[
\mathrm{min}\left\lbrace \frac{\mu(B\cap U^+)}{\mu(B)}, \frac{\mu(B\cap U^-)}{\mu(B)}\right\rbrace\leq C\frac{r}{\mu(\lambda B)}\Ha(S \cap \lambda B)
\]
for all balls $B=B(x, r)$ for which $\lambda B\subset \pi\tilde\phi^{-1}(-1/4, 1/4)$.
\end{lemma}
This formulation is essentially the same as the one used in \cite[Lemma 5.1]{LohvansuuRajala2018}, which is just an application of Theorems 6.2 and 1.1 of \cite{KorteLahti2014}. The same proof is valid here as well. Note that restricting the balls to $\phi^{-1}(-1/4, 1/4)$ ensures that $\partial U^+\cap \lambda B\subset S\cap\lambda B$.

Denote by $\Gamma$ the set of all paths $\gamma$ that connect $\pi\tilde\phi^{-1}(-1/8)$ to $\pi\tilde\phi^{-1}(1/8)$ inside $\pi\tilde\phi^{-1}(-1/4, 1/4)$.
\begin{corollary}\label{cor:polkupallot}
For every $n> \frac{1}{\delta}$ and $\gamma\in\Gamma$ there is a ball $B^n_\gamma$ that is centered on $\gamma$, has radius $\frac{1}{n}$ and satisfies
\[
\Ha(S\cap B_\gamma^n)\geq Cn\mu(B_\gamma^n)
\]
for some constant $C=C(T)$.
\end{corollary}
\begin{proof}
The proof is essentially contained in the discussion following Lemma 5.1 in \cite{LohvansuuRajala2018}. We sketch the idea here for completeness. Given a path $\gamma: [0, 1]\rightarrow T$ of $\Gamma$, we consider the balls $B_t:=B(\gamma(t), \frac{1}{2\lambda n})$, where $\lambda$ is as in Lemma \ref{lemma:riie}. We may assume that $|\gamma|$ is contained in $\pi\tilde\phi^{-1}[-1/8, 1/8]$, and therefore by the definition of $\delta$ each $B_t$ is contained in $\pi\tilde\phi^{-1}(-1/4, 1/4)$. Now the function
\[
\Phi: t\mapsto \frac{\mu(U^+\cap B_t)}{\mu(B_t)}
\]
vanishes when $t$ is near $0$ and is equal to $1$ when $t$ is near $1$. Pick 
\[
t_0:=\sup\{ t\in (0, 1)\, |\ \Phi(t)\leq 1/2 \}
\]
and choose $B_\gamma^n:=2\lambda B_{t_0}$. The lower bound on the measure of the boundary is then given by the relative isoperimetric inequality.
\end{proof}

Now let $\mathcal{F}_n$ be the collection of balls $B^n_\gamma$ that arise from the paths in $\Gamma$ as in Corollary \ref{cor:polkupallot} with $n$ fixed. Apply the $5r$ covering theorem on $\mathcal{F}_n$ to find a pairwise disjoint subcollection $\mathcal{G}_n$ with the property
\[
\bigcup_{B\in \mathcal{F}_n}B\subset \bigcup_{B\in\mathcal{G}_n}5B.
\]
Note that $\mathcal{G}_n$ must be finite due to the compactness of $T$. Write $\mathcal{G}_n=\{B_i\}_{i=1}^N$. Define a positive Borel function $\rho: T\rightarrow \R$ with 
\[
\rho:=n\sum_{i=1}^N\chi_{5B_i}.
\]
Let $\Omega$ be the open set that consists of the points that can be connected to $\pi\tilde\phi^{-1}(-1/8)$ by a rectifiable path inside $\pi\tilde\phi^{-1}(-1/4, 1/4)$. Define a function $\tilde\psi: T\rightarrow\R$ inside $\Omega$ with
\[
\tilde\psi(x):=\inf_{\gamma_x}\int_{\gamma_x}\rho\, ds,
\]
where the infimum is taken over all rectifiable paths $\gamma_x$ that connect $\pi\tilde\phi^{-1}(-1/8)$ to $x$ inside $\pi\tilde\phi^{-1}(-1/4, 1/4)$. Extend $\tilde\psi$ as zero to the rest of $T$. Finally, the desired competing admissible map $\psi: T\rightarrow \R/\Z$ is defined by
\[
\psi(x):=[\min\{ 1, \tilde\psi(x)\}].
\]
\begin{lemma}\label{lemma:liplocpsi}
The mapping $\psi$ is Lipschitz and $\rho$ is one of its upper gradients.
\end{lemma}
\begin{proof}
It is straightforward to prove that $\rho$ is an upper gradient of both $\tilde\psi$ and $\min\{1, \tilde\psi\}$ in $\Omega$, see e.g. \cite[Lemma 5.25]{BjornBjornNonlin}. Let $\gamma$ be a rectifiable path in $T$ that connects two points $x, y\in T$. The upper gradient inequality for the pair $(\psi, \rho)$ on $\gamma$ is immediate if $x, y\in\Omega$ and $|\gamma|\subset\Omega$, or if $\psi(x)=\psi(y)$. 

In order to prove the upper gradient inequality in the other possible situations we need to show that $\tilde\psi\geq 1$ on $\pi\tilde\phi^{-1}(1/8, 1/4)\cap\Omega$. To this end, let $\eta$ be a rectifiable path that connects $\pi\tilde\phi^{-1}(-1/8)$ to a point $x\in\pi\tilde\phi^{-1}(1/8, 1/4)$ inside $\pi\tilde\phi^{-1}(-1/4, 1/4)$. Then $\eta$ has a subpath $\eta'\in\Gamma$. Let $B_{\eta'}^n\in\mathcal{F}_n$ be the ball obtained by applying Corollary \ref{cor:polkupallot} on $\eta'$. Now
\[
\int_\eta\rho\, ds\geq\int_{|\eta'|}\rho\, d\Ha^1\geq n\sum_{i=1}^N\Ha^1(|\eta'|\cap 5B_i)\geq n\Ha^1(|\eta'|\cap B_{\eta'}^n)\geq 1,
\]
since $B_{\eta'}^n$ is covered by the balls $5B_i$. This holds for every connecting path $\eta$, which implies that $\tilde\psi(x)\geq 1$.

Next assume $x, y\in\Omega$ with $\tilde\psi(x), \tilde\psi(y)\in (0, 1)$ and $|\gamma|\not\subset\Omega$. Note that $\min\{1, \tilde\psi\}$ equals $0$ in $\pi\tilde\phi^{-1}(-1/4, -1/8)\cap\Omega$, since $\rho$ vanishes there. This means that there exist subpaths $\gamma_1=\gamma|_{[0, t_1]}$ and $\gamma_2=\gamma|_{[t_2, 1]}$ of $\gamma$ that satisfy $|\gamma_1|\cup|\gamma_2|\subset\Omega$ and $\psi(\gamma(t_1))=\psi(\gamma(t_2))=[0]$. Therefore
\begin{align*}
|\psi(x)-\psi(y)|&\leq|\psi(x)-\psi(\gamma(t_1))|+|\psi(\gamma(t_2))-\psi(y)| \\
&\leq\int_{\gamma_1}\rho\, ds+\int_{\gamma_2}\rho\, ds \\
&\leq\int_\gamma\rho \, ds.
\end{align*}
The same argument can be applied in the case of $x\in\Omega$, $y\not\in\Omega$. We omit the details.

The upper gradient inequality implies that $\psi$ is Lipschitz, since $T$ is geodesic and $\rho$ is bounded.
\end{proof}
\subsection{Degree of $\psi$}
In this subsection we prove that $\degg\psi=1$.

Pick a rectifiable degree 1 loop $\gamma$ and a point $a\in (1/8, 1/4)$. We may now assume that the endpoints of $\gamma$ are on $\pi\tilde\phi^{-1}(a)$. Since $T$ is geodesic and semilocally simply connected, we may assume that $\gamma$ has finite length. This, and moving the starting point if necessary, allows us to decompose $\gamma$ into 
\begin{equation}\label{eq:gammadecomposition}
\gamma=(\gamma_1*\eta_1)*\cdots*(\gamma_k*\eta_k),
\end{equation}
so that each $\gamma_i$ intersects $\pi\tilde\phi^{-1}(a)$ precisely at the endpoints, and none of the paths $\eta_i$ intersect $\pi\tilde\phi^{-1}(-a)$. 

For the next lemma we denote for brevity $\zeta:=\min\{1, \tilde\psi\}$.
\begin{lemma}\label{lemma:polkulemma1}
Let $\eta: [0, 1]\rightarrow \Omega$ be a rectifiable path. Suppose that the endpoints of $\zeta_*\eta$ belong to $\{0, 1\}$. Then $\psi_*\eta$ is loop-homotopic to $\zeta_*\eta(1)-\zeta_*\eta(0)$ times the standard generator of $\pi_1(\R/\Z, [0])$.
\end{lemma}
\begin{proof}
If the starting point is $0$ and the end point is $1$, the homotopy is given by $H: [0, 1]^2\rightarrow \R/\Z$, 
\[
H(s, t)=[s\zeta_*\eta(t)+(1-s)t].
\]
It is straightforward to check all the requirements. The other cases are similar.
\end{proof}
\begin{corollary}\label{cor:polkucor1}
The paths $\psi_*\eta_i$ and $\phi_*\eta_i$ are loop-contractible.
\end{corollary}
\begin{proof}
The endpoints of the path $\eta_i$ must be in the set $\pi\tilde\phi^{-1}(a)$. Since $\gamma$ has finite length, $\eta_i$ can be decomposed into
\[
\eta_i=\eta_i^1*\cdots*\eta_i^l,
\]
where the endpoints of each $\eta_i^j$ are in $\pi\tilde\phi^{-1}(a)$, and if $|\eta_i^j|\not\subset\Omega$, then there are no other intersections with $\pi\tilde\phi^{-1}(a)$.
 
Now if $\eta_i^j$ is contained in $\Omega$, Lemma \ref{lemma:polkulemma1} implies that it is loop-contractible. Otherwise $\psi_*\eta_i^j$ is already a constant path. Therefore $\psi_*\eta_i$ is loop-contractible as well. The path $\phi_*\eta_i$ cannot be surjective, so it is loop-contractible. 
\end{proof}

Let $\alpha: \R/\Z\rightarrow \R/\Z$ be the isomorphism $\alpha[x]=[x-a]$. Note that $\alpha_*\phi_*\gamma_i$ and $\psi_*\gamma_i$ are all loops with the same basepoint $[0]$.

Denote the domain of $\gamma_i$ by $[a_i, b_i]$. Let $\gamma'_i: [a_i, b_i]\rightarrow \R$ be the unique lift of $\alpha_*\phi_*\gamma_i$ for which $\gamma'_i(a_i)=0$. Further decompose each $\gamma_i$ into
\[
\gamma_i=\gamma_i^1*\gamma_i^2*\gamma_i^3,
\]
where $\gamma_i^1$ and $\gamma_i^3$ intersect $\pi\tilde\phi^{-1}(\pm a)$ exactly at their endpoints. 
\begin{lemma}\label{lemma:polkulemma2}
The lifted path $\gamma'_i$ intersects integer multiples of $\degg\phi$ exactly at its endpoints. In particular $\gamma'_i(b_i)=\pm\degg\phi$ or $\gamma'_i(b_i)=0$. Moreover, $\gamma_i^1$ (respectively $\gamma_i^3$) is contained in $\Omega$ if and only if $\gamma'_i$ is negative in a neighborhood of $a_i$ ($\gamma'_i\leq\gamma'(b_i)$ in a neighborhood of $b_i$).
\end{lemma}
\begin{proof}
Let $\tilde\gamma_i: [a_i, b_i]\rightarrow\tilde T$ be the lift of $\gamma_i$ that satisfies $\tilde\phi_*\tilde\gamma_i(a_i)=a$. Then due to uniqueness of lifts we have $\gamma'_i=\tilde\phi_*\tilde\gamma_i-a$. Since $\tilde\phi$ is a lift of $\phi$, we have $\tilde\phi\circ\tau^k=k\cdot\degg\phi+\tilde\phi$ for any integer $k$. It follows that $\gamma_i'(t)=k\cdot\degg\phi$ if and only if $\tilde\phi(\tau^{-k}(\tilde\gamma_i(t)))=a$, which can be combined with the lifting property $\pi_*\tilde\gamma_i=\gamma_i$ to conclude that $\gamma_i'(t)$ equals an integer multiple of $\degg\phi$ if and only if $\gamma_i(t)\in\pi\tilde\phi^{-1}(a)$. By construction the latter happens if and only if $t$ equals either endpoint of $[a_i, b_i]$. This proves the first assertion of the lemma.

The definitions of $\gamma_i^1, \gamma_i^3$ and $\Omega$ imply that these paths are contained in $\Omega$ if and only if they are contained in $\pi\tilde\phi^{-1}[-a, a]$. Therefore $\gamma_i^1$ is contained in $\Omega$ if and only if the part of $\tilde\gamma_i$ corresponding to $\gamma_i^1$ is contained in $\tilde\phi^{-1}(k\cdot\degg\phi + [-a, a])$ for some fixed integer $k$. This $k$ must be $0$, since we chose $\tilde\phi_*\tilde\gamma_i(a_i)=a$. Thus $\gamma'_i=\tilde\phi_*\tilde\gamma_i-a$ is negative in a neighborhood of $a_i$ if and only if $\gamma_i^1$ is contained in $\Omega$. The path $\gamma_i^3$ can be treated similarly.
\end{proof}
\begin{corollary}\label{cor:polkucor2}
The paths $\alpha_*\phi_*\gamma_i$ and $\degg \phi\cdot\psi_*\gamma_i$ are loop-homotopic.
\end{corollary}
\begin{proof}
We need to check four different cases, corresponding to $\gamma_i^1$ and $\gamma_i^3$ being or not being contained in $\Omega$. The proofs are essentially the same, so we write down only one of them. 

Assume that $\gamma_i^1$ is not contained in $\Omega$ but $\gamma_i^3$ is. Then $\psi_*\gamma_i^1$ is a constant path, and $\psi_*\gamma_i^3$ is loop-homotopic to the standard generator by Lemma \ref{lemma:polkulemma1}. Arguing exactly as in the proof of Corollary \ref{cor:polkucor1}, we see that $\psi_*\gamma_i^2$ is loop-contractible. Therefore $\psi_*\gamma_i$ is loop-homotopic to the standard generator.

By Lemma \ref{lemma:polkulemma2} the lift $\gamma_i'$ satisfies $\gamma_i'(a_i)=0$ and $\gamma_i'(b_i)=\pm\degg\phi$ or $\gamma_i'(b_i)=0$. We also find that $\gamma_i'$ is positive in a neighborhood of $a_i$, and less than $\gamma_i'(b_i)$ in a neighborhood of $b_i$. Combining these gives $\gamma_i'(b_i)=\degg\phi$, which means precisely that $\alpha_*\phi_*\gamma_i$ is loop-homotopic to $\degg\phi$ times the standard generator.
\end{proof}
Applying Corollaries \ref{cor:polkucor1} and \ref{cor:polkucor2} to the decomposition \eqref{eq:gammadecomposition} yields 
\[
\alpha_*\phi_*\gamma\simeq \degg\phi\cdot\psi_*\gamma.
\]
Now by applying the identity $\degg(\alpha\circ\phi)=\degg\phi$, we see that $\psi_*\gamma$ is loop-homotopic to the standard generator. Therefore $\degg\psi=1$ and the proof of Proposition \ref{prop:degree} is finished.


\bibliographystyle{plain}
\bibliography{modulijuttubib}

\begin{thebibliography}{10}

\bibitem{BjornBjornNonlin}
Anders Bj\"{o}rn and Jana Bj\"{o}rn.
\newblock {\em Nonlinear potential theory on metric spaces}, volume~17 of {\em
  EMS Tracts in Mathematics}.
\newblock European Mathematical Society (EMS), Z\"{u}rich, 2011.

\bibitem{BjornBjorn2018}
Anders Bj\"{o}rn and Jana Bj\"{o}rn.
\newblock Local and semilocal {P}oincar\'{e} inequalities on metric spaces.
\newblock {\em J. Math. Pures Appl. (9)}, 119:158--192, 2018.

\bibitem{Cheeger1999}
J.~Cheeger.
\newblock Differentiability of {L}ipschitz functions on metric measure spaces.
\newblock {\em Geom. Funct. Anal.}, 9(3):428--517, 1999.

\bibitem{FreedmanHe1991}
Michael~H. Freedman and Zheng-Xu He.
\newblock Divergence-free fields: energy and asymptotic crossing number.
\newblock {\em Ann. of Math. (2)}, 134(1):189--229, 1991.

\bibitem{Fuglede1957}
Bent Fuglede.
\newblock Extremal length and functional completion.
\newblock {\em Acta Math.}, 98:171--219, 1957.

\bibitem{Gehring1962}
F.~W. Gehring.
\newblock Extremal length definitions for the conformal capacity of rings in
  space.
\newblock {\em Michigan Math. J.}, 9:137--150, 1962.

\bibitem{HKST}
Juha Heinonen, Pekka Koskela, Nageswari Shanmugalingam, and Jeremy~T. Tyson.
\newblock {\em Sobolev spaces on metric measure spaces, an approach based on
  upper gradients}, volume~27 of {\em New Mathematical Monographs}.
\newblock Cambridge University Press, Cambridge, 2015.

\bibitem{Ikonen2019}
Toni Ikonen.
\newblock Uniformization of metric surfaces using isothermal coordinates,
  preprint, ar{X}iv:1909.09113.

\bibitem{JonesLahti2019}
Rebekah Jones and Panu Lahti.
\newblock Duality of moduli and quasiconformal mappings in metric spaces,
  preprint, ar{X}iv:1905.02873.

\bibitem{KallunkiShan2001}
Sari Kallunki and Nageswari Shanmugalingam.
\newblock Modulus and continuous capacity.
\newblock {\em Ann. Acad. Sci. Fenn. Math.}, 26(2):455--464, 2001.

\bibitem{KorteLahti2014}
Riikka Korte and Panu Lahti.
\newblock Relative isoperimetric inequalities and sufficient conditions for
  finite perimeter on metric spaces.
\newblock {\em Ann. Inst. H. Poincar\'e Anal. Non Lin\'eaire}, 31(1):129--154,
  2014.

\bibitem{LohvansuuRajala2018}
Atte Lohvansuu and Kai Rajala.
\newblock Duality of moduli in regular metric spaces.
\newblock {\em Indiana Univ. Math. J.}, to appear.

\bibitem{Rajala2017}
Kai Rajala.
\newblock Uniformization of two-dimensional metric surfaces.
\newblock {\em Invent. Math.}, 207(3):1301--1375, 2017.

\bibitem{RickmanQR}
Seppo Rickman.
\newblock {\em Quasiregular mappings}, volume~26 of {\em Ergebnisse der
  Mathematik und ihrer Grenzgebiete (3) [Results in Mathematics and Related
  Areas (3)]}.
\newblock Springer-Verlag, Berlin, 1993.

\bibitem{Vaisala}
Jussi V\"ais\"al\"a.
\newblock {\em Lectures on {$n$}-dimensional quasiconformal mappings}.
\newblock Lecture Notes in Mathematics, Vol. 229. Springer-Verlag, Berlin-New
  York, 1971.

\bibitem{Ziemer1967}
William~P. Ziemer.
\newblock Extremal length and conformal capacity.
\newblock {\em Trans. Amer. Math. Soc.}, 126:460--473, 1967.

\end{thebibliography}
\vspace{1em}
\noindent
Department of Mathematics and Statistics, University of Jyv\"askyl\"a, P.O.
Box 35 (MaD), FI-40014, University of Jyv\"askyl\"a, Finland.\\

\emph{E-mail:} \settowidth{\hangindent}{\emph{aaaaaaaaa}}\textbf{atte.s.lohvansuu@jyu.fi} 
\end{document}